\documentclass{elsart}               

\usepackage{graphicx}

\usepackage{amssymb}      
\usepackage{bm}           
\usepackage{amsfonts}     

\usepackage{algorithm}
\usepackage{algpseudocode}
\usepackage{booktabs}
\usepackage{color}

\def\adj{\mathop{\rm adj}\nolimits}
\def\pdet{\mathop{\rm pdet}\nolimits}
\def\ind{\mathop{\rm ind}\nolimits}
\def\diag{\mathop{\rm diag}\nolimits}
\def\rank{\mathop{\rm rank}\nolimits}

\begin{document}

	\begin{frontmatter}
		
		\title{Determinant Dynamics under Low-Rank Perturbations:
			A Unified Framework for Singular Systems\thanksref{footnoteinfo}} 
		
		\thanks[footnoteinfo]{Corresponding author R.~Vrabel.}
		
		\author[RV]{Robert Vrabel}\ead{robert.vrabel@stuba.sk}    
		
		\address[RV]{Slovak University of Technology in Bratislava, Institute of Applied Informatics, Automation and Mechatronics,  Bottova 25,  917 24 Trnava, Slovakia}  

		
		\begin{abstract}
			This paper develops a unified analytical framework for determinant
			identities under finite-rank perturbations of square matrices that
			remains valid without invertibility assumptions. In contrast to
			classical inverse-based formulations, the approach is based on an
			adjugate-driven additive representation, which extends naturally to
			singular matrices and yields explicit, non-asymptotic formulas.
			Building on this representation, we derive recursive and multiplicative
			expressions describing the evolution of determinant and log-determinant
			quantities under successive rank-one updates. These results reveal a
			structural interpretation in which determinant-based quantities evolve
			as cumulative measures of independent directions, providing a precise
			decomposition of incremental contributions.
			To address the singular case, we develop a systematic extension based on
			the Drazin inverse and the pseudodeterminant, leading to closed-form 
			identities that isolate the contribution of the nonzero spectrum. In
			particular, we obtain a generalized determinant formula that can be
			viewed as a singular counterpart of the matrix determinant lemma.
			The spectral impact of low-rank perturbations is analyzed, yielding
			explicit conditions governing eigenvalue shifts and stability
			preservation.
			
			The proposed framework establishes a direct analytical link between
			matrix perturbation theory and system-theoretic concepts. In particular,
			we show that the pseudodeterminant of controllability Gramians admits a
			multiplicative decomposition that explicitly quantifies the incremental
			expansion of the reachable subspace under successive inputs. This leads
			to a unified interpretation of information accumulation, uncertainty
			reduction, and reachability in both full-rank and rank-deficient linear
			systems.
		\end{abstract} 
		
	\end{frontmatter}

	\section{Introduction}
	
	Low-rank perturbations of matrices arise naturally in control theory,
	estimation, and optimization. They appear, for example, in recursive
	covariance updates, Kalman filtering, experiment design, and in the
	analysis of controllability and observability Gramians. In many of
	these applications, the effect of successive rank-one updates on
	determinants and eigenvalues plays a central role.
	
	Classical determinant identities, such as the matrix determinant lemma,
	are typically formulated for rank-one updates under invertibility
	assumptions \cite{HornJohnson1985,Harville1997}. While these formulas
	are well understood in the nonsingular case, their direct extension to
	singular matrices or to finite-rank perturbations is not straightforward.
	In particular, multiplicative representations based on matrix inverses
	break down when singularity is present.
	
	Low-rank matrix structures also play an increasingly important role in
	modern control and estimation. In recursive estimation and Kalman
	filtering, covariance and information matrices are updated sequentially
	as new measurements are incorporated \cite{Bierman1974,Verhaegen1986}.
	Such factored and square-root formulations improve numerical stability
	and computational efficiency. More recently, low-rank representations
	have gained renewed attention in high-dimensional and data-driven
	settings, where they enable scalable filtering and inference
	\cite{Schmidt2023,Li2024}. These approaches exploit the fact that
	covariance updates often evolve in low-dimensional subspaces.
	
	In parallel, log-determinant functionals have been widely used as
	measures of information, uncertainty, and rank surrogates in optimization
	and learning \cite{Boyd2004,Kang2015}. In control, related structures
	appear in reachability analysis through controllability Gramians,
	which describe the accumulation of reachable directions and the
	geometry of the reachable set \cite{Kailath1980,Gough2015}.
	Distributed and networked estimation further highlights the importance
	of structured covariance updates \cite{OlfatiSaber2007}.
	
	Despite these developments, existing works primarily focus on numerical,
	algorithmic, or application-specific aspects of low-rank updates.
	In contrast, the present paper develops explicit analytical identities
	that reveal the underlying structure of determinant and log-determinant
	evolution under finite-rank perturbations.
	
	The aim of this paper is to develop determinant identities for
	finite-rank perturbations of arbitrary square matrices, without
	requiring invertibility. The approach is based on an additive
	representation involving the adjugate matrix, which remains valid
	in the singular case and leads naturally to recursive formulas
	describing determinant evolution.
	
	A key feature of the proposed framework is that it provides a
	unified interpretation of determinant and log-determinant quantities
	as cumulative measures associated with individual update directions.
	This viewpoint is particularly relevant in control and estimation,
	where log-determinant expressions quantify uncertainty, information
	content, and the volume of reachable or observable sets.
	
	In addition, the singular case is addressed using the Drazin inverse
	introduced in \cite{Drazin01081958} together with the pseudodeterminant,
	which provide a consistent extension of determinant identities beyond
	the nonsingular setting. The resulting formulas capture both spectral
	changes and structural properties of the underlying system. Among
	generalized inverses, the Moore--Penrose pseudoinverse is perhaps the
	most widely used; both the Moore--Penrose and the Drazin inverse can be
	understood in terms of canonical matrix decompositions, the former
	based on the singular value decomposition and the latter on the
	Jordan canonical form.
	
	The developed identities are applied to several problems in control
	and estimation. In particular, we analyze covariance updates in
	recursive estimation, derive explicit expressions for log-determinant
	evolution in Kalman-type filtering, and study determinant-based
	characterizations of controllability Gramians and reachable sets.
	These applications illustrate that determinant-based representations
	provide a common framework for analyzing information accumulation,
	uncertainty reduction, and reachability in linear systems. Moreover,
	the proposed framework naturally extends to singular settings, where
	the pseudodeterminant provides a consistent measure of the volume of
	the reachable subspace, allowing a unified treatment of both full-rank
	and rank-deficient systems.

	\noindent
	\textit{Main contributions.}
	The main contributions of this paper are as follows:
	\begin{enumerate}
		\item We introduce an adjugate-based framework for finite-rank perturbations
		that removes invertibility assumptions and provides a unified algebraic
		representation valid for both nonsingular and singular matrices.
		
		\item We uncover a \emph{determinant dynamics} under successive rank-one updates,
		showing that determinant and log-determinant quantities admit explicit
		additive and multiplicative decompositions. This reveals a previously
		unexploited structural interpretation in which each update contributes
		an incremental term associated with a specific direction.
		
		\item We extend this structure to singular systems using the Drazin inverse
		and the pseudodeterminant, obtaining multiplicative formulas that remain
		well-defined in the presence of rank deficiency and isolate the contribution
		of the nonzero spectrum.
		
		\item We establish that these representations induce a \emph{multiplicative
			geometry} of low-rank updates: determinant-based quantities evolve as
		cumulative measures of independent directions, allowing explicit
		characterization of growth, decay, and conditioning effects.
		
		\item We demonstrate that this framework provides a direct analytical link
		between matrix perturbation theory and system-theoretic concepts. In particular,
		we show that the pseudodeterminant of controllability Gramians admits a
		multiplicative decomposition that quantifies the expansion of the reachable
		subspace, thereby yielding a unified interpretation of information accumulation,
		uncertainty reduction, and reachability in linear systems.
	\end{enumerate}
	
	The paper is organized according to a progression from general
	algebraic identities to nonsingular and singular formulations,
	and finally to system-theoretic applications.
	The results developed in the paper should be viewed as complementary
	layers of the same determinant-evolution mechanism under low-rank
	updates.
	Theorem~\ref{thm:rank1} provides the basic adjugate-based identity,
	which is valid without invertibility assumptions and underlies the
	recursive determinant dynamics of Theorem~\ref{thm:main}. When the
	intermediate matrices are nonsingular, this mechanism admits the
	multiplicative inverse-based formulation of Theorem~\ref{thm:logdet}
	and Corollary~\ref{cor:log}.
	For singular matrices with semisimple zero eigenvalue,
	Definition~\ref{def:singular_structure}, Lemma~\ref{lem:block_drazin},
	and Lemma~\ref{lem:projected_columns} provide the structural tools
	needed to pass to the Drazin inverse and the pseudodeterminant. This
	leads to the singular multiplicative identity in
	Theorem~\ref{thm:pdet_drazin}. Theorem~\ref{thm:regularized_det}
	provides a complementary regularized version of the same structure,
	which is particularly useful when singular matrices arise as limits of
	nonsingular ones.
	These layers reappear in the applications: the nonsingular formulas are
	used in covariance and Kalman-type updates, whereas the regularized
	pseudodeterminant viewpoint is essential in the controllability-Gramian
	analysis of Theorem~\ref{thm:gramian_growth_pdet}. Thus, the adjugate,
	inverse, Drazin, and regularized formulations are different
	manifestations of a single structural principle adapted to general,
	nonsingular, and singular settings.

	\section{Notations and preliminaries}\label{sec:Notations-and-Preliminaries}
	
	Let $H \in {\mathbb R}^{n \times n}$ be a square matrix. For matrices
	$U, V \in {\mathbb R}^{n \times r}$ we write
	\(
	UV^T = \sum_{i=1}^r u_i v_i^T,
	\)
	where $u_i, v_i \in {\mathbb R}^n$ denote the columns of $U$ and $V$.
	Each term $u_i v_i^T$ represents a rank-one update acting along the
	direction $u_i$ with sensitivity defined by $v_i$.
	
	Matrices of the form
	\(
	H + UV^T
	\)
	will be referred to as finite-rank perturbations of $H$. It is often
	convenient to interpret such perturbations as sequences of rank-one
	updates. Given vectors $u_i, v_i \in {\mathbb R}^n$, $i=1,\dots,r$, we define
	\(
	\Delta_0 = 0,
	\
	\Delta_k = \sum_{j=1}^k u_j v_j^T,
	\quad k=1,\dots,r.
	\)
	Then $H+\Delta_k$ represents the result of $k$ successive rank-one
	updates, and in particular $\Delta_r = UV^T$.
	
	The adjugate of a matrix $A$ is denoted by $\adj(A)$ and satisfies
	\(
	A \adj(A) = \adj(A) A = \det(A) I.
	\)
	In the singular case, we make use of the pseudodeterminant and the
	Drazin inverse. The pseudodeterminant of $H$, denoted by $\pdet(H)$,
	is defined as the product of all nonzero eigenvalues of $H$, counted
	with multiplicity. This quantity extends the determinant to singular
	matrices by capturing the contribution of the nonzero spectrum.
	
	The Drazin inverse $H^D$ is defined for arbitrary square matrices and
	is characterized by
	\[
	H H^D = H^D H, \ \
	H^D H H^D = H^D, \ \
	H^{k+1} H^D = H^k,
	\]
	where $k = \ind(H)$ is the index of $H$, i.e., the size of the largest
	Jordan block associated with the eigenvalue zero.
	
	The matrix
	\(
	P_0 = I - H H^D
	\)
	is the spectral projector onto the generalized nullspace of $H$. It acts
	as the identity on the generalized nullspace and annihilates all
	components corresponding to nonzero eigenvalues.
	
	The following structural description will be used throughout the paper.
	
	\begin{defn}\label{def:singular_structure}
		Let $H \in {\mathbb R}^{n \times n}$. The matrix $H$ is said to have a
		semisimple zero eigenvalue if
		\(
		\ind(H)=1.
		\)
		In this case, there exists a nonsingular matrix $S$ such that
		\[
		S^{-1}HS=
		\left[
		\begin{array}{cc}
			J & 0\\
			0 & 0
		\end{array}
		\right],
		\]
		where $J \in {\mathbb R}^{q \times q}$ is nonsingular and $q+\nu=n$.
		Moreover,
		\[
		\pdet(H)=\det(J),
		\]
		and the Drazin inverse and the spectral projector are given by
		\[
		H^D =
		S
		\left[
		\begin{array}{cc}
			J^{-1} & 0\\
			0 & 0
		\end{array}
		\right]
		S^{-1},
		\
		P_0 =
		S
		\left[
		\begin{array}{cc}
			0 & 0\\
			0 & I_\nu
		\end{array}
		\right]
		S^{-1}.
		\]
	\end{defn}
	
	The following identity will play a central role in the subsequent
	analysis.
	
	\begin{thm}\label{thm:rank1}
		Let $H \in {\mathbb R}^{n \times n}$ and let
		$u, v \in {\mathbb R}^n$. Then
		\[
		\det(H + uv^T)
		=
		\det(H) + v^T \adj(H) u.
		\]
	\end{thm}
	
	Theorem \ref{thm:rank1} can be viewed as a form of the matrix
	determinant lemma that remains valid without any invertibility
	assumption; see, e.g., \cite{Vrabel2016}.

	\section{Finite-rank perturbations and determinant dynamics}\label{sec:Finite-rank perturbations and determinant dynamics}
	
	\begin{thm}\label{thm:main}
		Let $H \in {\mathbb R}^{n \times n}$ and let
		$u_i, v_i \in {\mathbb R}^n$, $i=1,\dots,r$.
		Then for every $k=1,\dots,r$ we have
		\begin{equation}
			\det(H + \Delta_k)
			=
			\det(H)
			+
			\sum_{i=1}^k
			v_i^T \adj(H + \Delta_{i-1}) u_i.
			\label{eq:thm_main}
		\end{equation}
		Moreover, the sequence
		\(
		D_k = \det(H + \Delta_k), \ k=0,1,\dots,r,
		\)
		satisfies
		\[
		D_k = D_{k-1} + v_k^T \adj(H + \Delta_{k-1}) u_k,
		\ k=1,\dots,r.
		\]
	\end{thm}
	
	\begin{pf}
		We proceed by induction on $k$.
		
		For $k=1$, we have
		\(
		\Delta_1 = u_1 v_1^T,
		\)
		and Theorem \ref{thm:rank1} gives
		\(
		\det(H + \Delta_1)
		=
		\det(H) + v_1^T \adj(H) u_1.
		\)
		Thus the result holds for $k=1$.
		
		Assume that for some $k-1 \ge 1$ we have
		\[
		\det(H + \Delta_{k-1})
		=
		\det(H)
		+
		\sum_{i=1}^{k-1}
		v_i^T \adj(H + \Delta_{i-1}) u_i.
		\]
		Since
		\(
		\Delta_k = \Delta_{k-1} + u_k v_k^T,
		\)
		we can write
		\(
		H + \Delta_k = (H + \Delta_{k-1}) + u_k v_k^T.
		\)
		Applying Theorem \ref{thm:rank1} to the matrix $H + \Delta_{k-1}$,
		we obtain
		\[
		\det(H + \Delta_k)
		=
		\det(H + \Delta_{k-1})
		+
		v_k^T \adj(H + \Delta_{k-1}) u_k.
		\]
		Substituting the induction hypothesis yields (\ref{eq:thm_main}).
		This proves the first statement.
		The second statement follows immediately from the definition
		$D_k = \det(H + \Delta_k)$ and the identity above.
	\end{pf}
	
	Theorem \ref{thm:main} shows that the determinant under
	finite-rank perturbations admits both a global representation
	and a recursive update formula. The latter can be interpreted
	as a discrete evolution law for the determinant.
	
	\section{Log-determinant representation}\label{sec:Log-determinant representation}
	
	The logarithm of the determinant plays a fundamental role in
	control, estimation, and information theory. In particular,
	for a positive definite matrix $H$, the quantity
	\(
	\log \det(H)
	\)
	admits several important interpretations and possesses rich
	analytical structure.
	Geometrically, $\det(H)$ is proportional to the volume of the
	ellipsoid
	\(
	\{x \in {\mathbb R}^n : x^T H^{-1} x \le 1\},
	\)
	and therefore $\log \det(H)$ measures this volume on a logarithmic
	scale. This logarithmic transformation converts multiplicative
	volume changes into additive quantities, which is particularly
	useful for analyzing incremental updates.
	
	From an information-theoretic perspective, if $H$ represents a
	covariance matrix of a Gaussian random vector, then
	\(
	\frac{1}{2} \log \det(H)
	\)
	is proportional to the differential entropy; see, e.g., \cite{cover2006elements}.
	In this context, $\log \det(H)$ quantifies the uncertainty or information
	content of the system.
	
	Beyond these classical interpretations, the log-de\-ter\-mi\-nant also
	plays a central role in modern optimization and data science.
	In particular, it is a strictly concave function on the cone of
	positive definite matrices and serves as a self-concordant barrier
	for semidefinite programming. Moreover, it appears in likelihood-based
	inference for high-dimensional Gaussian models and as a smooth surrogate
	for rank in low-rank approximation problems.
	In estimation and control, log-determinant criteria arise naturally
	in optimal experiment design, sensor placement, and Kalman filtering.
	Maximizing $\log \det(H)$ corresponds to maximizing the information
	gained from measurements, while minimizing it corresponds to reducing
	uncertainty.
	
	These interpretations share a common structural feature: the
	log-determinant transforms multiplicative effects into additive
	contributions, making it particularly suitable for analyzing
	sequential updates.
	For these reasons, it is of interest to obtain explicit representations
	of $\log \det(H + UV^T)$ under structured perturbations. In particular,
	such representations allow one to decompose the evolution of
	$\log \det(H)$ into incremental contributions associated with
	individual update directions, providing a transparent analytical
	description of information accumulation and structural change.
	
	\begin{thm}\label{thm:logdet}
		Assume that the matrices
		\(
		H + \Delta_k, \ k=0,1,\dots,r-1,
		\)
		are nonsingular  (in particular, $H$ is nonsingular). Then
		\[
		\det(H + \Delta_r)
		=
		\det(H)
		\prod_{i=1}^r
		\left(
		1 + v_i^T (H + \Delta_{i-1})^{-1} u_i
		\right).
		\]
	\end{thm}
	
	\begin{pf}
		For each $i=1,\dots,r$, put
		\(
		M_{i-1}=H+\Delta_{i-1}.
		\)
		Then, by definition of $\Delta_i$, we have
		\(
		\Delta_i=\Delta_{i-1}+u_i v_i^T,
		\)
		and therefore
		\[
		H+\Delta_i
		=
		(H+\Delta_{i-1})+u_i v_i^T
		=
		M_{i-1}+u_i v_i^T.
		\]
		Since $M_{i-1}=H+\Delta_{i-1}$ is nonsingular by assumption, we may apply
		Theorem \ref{thm:rank1} to the matrix $M_{i-1}$. This gives
		\[
		\det(M_{i-1}+u_i v_i^T)
		=
		\det(M_{i-1})+v_i^T \adj(M_{i-1})u_i.
		\]
		Because $M_{i-1}$ is nonsingular, we also have the standard identity
		\(
		\adj(M_{i-1})=\det(M_{i-1})\,M_{i-1}^{-1}.
		\)
		Hence
		\[
		v_i^T \adj(M_{i-1})u_i
		=
		v_i^T \left(\det(M_{i-1})\,M_{i-1}^{-1}\right)u_i
		=	\det(M_{i-1})\,v_i^T M_{i-1}^{-1}u_i.
		\]
		Substituting this into the previous equality, we obtain
		\[
		\det(M_{i-1}+u_i v_i^T)
		=
		\det(M_{i-1})
		+
		\det(M_{i-1})\,v_i^T M_{i-1}^{-1}u_i.
		\]
		Factoring out $\det(M_{i-1})$, it follows that
		\[
		\det(M_{i-1}+u_i v_i^T)
		=
		\det(M_{i-1})
		\left(
		1+v_i^T M_{i-1}^{-1}u_i
		\right).
		\]
		Recalling that
		\(
		M_{i-1}+u_i v_i^T=H+\Delta_i
		\)
		and
		\(
		M_{i-1}=H+\Delta_{i-1},
		\)
		we arrive at
		\[
		\det(H+\Delta_i)
		=
		\det(H+\Delta_{i-1})
		\left(
		1+v_i^T (H+\Delta_{i-1})^{-1}u_i
		\right)
		\]
		for every $i=1,\dots,r$.
		Now we apply this identity successively for $i=1,2,\dots,r$.
		For $i=1$,
		\[
		\det(H+\Delta_1)
		=
		\det(H)
		\left(
		1+v_1^T (H+\Delta_0)^{-1}u_1
		\right).
		\]
		Since $\Delta_0=0$, this becomes
		\(
		\det(H+\Delta_1)
		=
		\det(H)
		\left(
		1+v_1^T H^{-1}u_1
		\right).
		\)
		
		For $i=2$,
		\(
		\det(H+\Delta_2)
		=
		\det(H+\Delta_1)
		\left(
		1+v_2^T (H+\Delta_1)^{-1}u_2
		\right).
		\)
		Substituting the expression for $\det(H+\Delta_1)$, we get
		\[
		\det(H+\Delta_2)
		=
		\det(H)
		\left(
		1+v_1^T (H+\Delta_0)^{-1}u_1
		\right)
		\left(
		1+v_2^T (H+\Delta_1)^{-1}u_2
		\right).
		\]
		Continuing in the same way, after $r$ steps we obtain
		\[
		\det(H+\Delta_r)
		=
		\det(H)
		\prod_{i=1}^r
		\left(
		1+v_i^T (H+\Delta_{i-1})^{-1}u_i
		\right).
		\]
		This completes the proof.
	\end{pf}
	\begin{cor}
		Assume that the matrices $H + \Delta_k$, $k=0,1,\dots,r$, are
		nonsingular and that
		\(
		\det(H + \Delta_k) > 0, \ k=0,1,\dots,r.
		\)
		Then
		\[
		\log \det(H + \Delta_r)
		=
		\log \det(H)
		+
		\sum_{i=1}^r
		\log
		\left(
		1 + v_i^T (H + \Delta_{i-1})^{-1} u_i
		\right).
		\]
		\label{cor:log}
	\end{cor}
	\begin{exmp}
		\label{ex:contribution}
		Let
		\(
		H = I_n
		\)
		and consider a sequence of rank-one updates
		\(
		\Delta_k = \sum_{i=1}^k u_i u_i^T.
		\)
		Then
		\(
		H + \Delta_k = I_n + \sum_{i=1}^k u_i u_i^T.
		\)
		Since $(I_n + \Delta_{i-1})^{-1}$ is positive definite, we have
		\(
		u_i^T (I_n + \Delta_{i-1})^{-1} u_i \ge 0.
		\)
		Hence, by Corollary~\ref{cor:log},
		\[
		\log \det(H + \Delta_k)
		=
		\sum_{i=1}^k
		\log
		\left(
		1 + u_i^T (I_n + \Delta_{i-1})^{-1} u_i
		\right),
		\]
		which is a nondecreasing sequence.
		Moreover, each term
		\[
		\log
		\left(
		1 + u_i^T (I_n + \Delta_{i-1})^{-1} u_i
		\right)
		\]
		represents the incremental contribution of the vector $u_i$.
		
		To make this interpretation precise, observe that the matrix
		$I_n + \Delta_{i-1}$ is symmetric positive definite. Hence it admits
		an eigenvalue decomposition
		\(
		I_n + \Delta_{i-1} = Q \Lambda Q^T,
		\)
		where $\Lambda = \diag(\lambda_1,\dots,\lambda_n)$ with
		$\lambda_j > 0$, and $Q$ is orthogonal.
		
		Then
		\[
		u_i^T (I_n + \Delta_{i-1})^{-1} u_i
		=
		\sum_{j=1}^n \frac{\alpha_j^2}{\lambda_j},
		\]
		where $\alpha_j$ are the coordinates of $u_i$ in the eigenbasis
		of $I_n + \Delta_{i-1}$.
		This expression shows that the contribution of $u_i$ depends on
		its alignment with the eigenvectors of $I_n + \Delta_{i-1}$, weighted
		by the corresponding eigenvalues.
		If $u_i$ lies predominantly in directions associated with large
		eigenvalues $\lambda_j$ (i.e., directions already well represented
		in $\Delta_{i-1}$), then the ratios $\alpha_j^2 / \lambda_j$
		are small, and consequently
		\(
		u_i^T (I_n + \Delta_{i-1})^{-1} u_i
		\)
		is small. In this case, the increment
		\[
		\log\left(1 + u_i^T (I_n + \Delta_{i-1})^{-1} u_i\right)
		\]
		is also small.
		
		On the other hand, if $u_i$ has a significant component in directions
		where $\lambda_j$ are close to $1$, that is, directions not yet well
		represented in $\Delta_{i-1}$, then the quantity
		\[
		u_i^T (I_n + \Delta_{i-1})^{-1} u_i
		\]
		is larger, leading to a larger increase in the log-determinant.
		Thus, the magnitude of the increment reflects how much new
		independent direction is introduced by $u_i$. Directions already
		present in $\Delta_{i-1}$ contribute less, while new directions
		contribute more.
		This shows that the log-determinant naturally quantifies
		the accumulation of independent directions or information.
	\end{exmp}
	
	The above representation shows that the log-determinant evolves
	additively under finite-rank perturbations. Each term provides a
	quantitative measure of how much the current update increases
	the volume or information associated with the matrix.

	\section{A singular extension with the Drazin inverse}\label{sec:A singular extension via the Drazin inverse}
	In the nonsingular case, the matrix determinant lemma
	provides the formula
	\[
	\det(H+UV^T)
	=
	\det(H)\det(I+V^T H^{-1} U).
	\]
	This identity follows by factoring out $H$ as
	\(
	H+UV^T = H\bigl(I+H^{-1}UV^T\bigr)
	\)
	and applying the multiplicativity of the  Sylvester's determinant identity
	$\det(I+AB)=\det(I+BA)$ (see, e.g., \cite{HornJohnson1985,AKRITAS1996585} for a more general formulation and context).
	It is therefore natural to ask whether a similar identity
	can be extended to singular matrices by replacing the inverse
	with a generalized inverse such as the Drazin inverse.
	However, such a direct extension does not hold in general.
	If $\det(H)=0$, then identities of the above form
	are either false or degenerate unless additional structure
	is imposed.
	Instead, meaningful extensions can be obtained by replacing
	the determinant with the pseudodeterminant or by using
	regularized determinant expressions.

	\begin{lem}\label{lem:block_drazin}
		Let $H \in {\mathbb R}^{n \times n}$ satisfy $\ind(H)=1$, and let
		\[
		S^{-1}HS=
		\left[
		\begin{array}{cc}
			J & 0\\
			0 & 0
		\end{array}
		\right]
		\]
		with $J \in {\mathbb R}^{q \times q}$ nonsingular and $q+\nu=n$.
		Then
		\[
		H^D=
		S
		\left[
		\begin{array}{cc}
			J^{-1} & 0\\
			0 & 0
		\end{array}
		\right]
		S^{-1}
		\]
		and
		\[
		P_0=
		S
		\left[
		\begin{array}{cc}
			0 & 0\\
			0 & I_\nu
		\end{array}
		\right]
		S^{-1}.
		\]
	\end{lem}
	
	\begin{pf}
		Define
		\[
		\widetilde H=
		\left[
		\begin{array}{cc}
			J & 0\\
			0 & 0
		\end{array}
		\right]
		\qquad\mbox{and}\qquad
		\widetilde X=
		\left[
		\begin{array}{cc}
			J^{-1} & 0\\
			0 & 0
		\end{array}
		\right].
		\]
		Since
		\(
		H=S\widetilde H S^{-1},
		\)
		it is enough to verify that \(\widetilde X\) satisfies the defining
		properties of the Drazin inverse of \(\widetilde H\).
		First,
		\[
		\widetilde H \widetilde X
		=
		\left[
		\begin{array}{cc}
			J & 0\\
			0 & 0
		\end{array}
		\right]
		\left[
		\begin{array}{cc}
			J^{-1} & 0\\
			0 & 0
		\end{array}
		\right]
		=
		\left[
		\begin{array}{cc}
			I_q & 0\\
			0 & 0
		\end{array}
		\right]
		=
		\widetilde X \widetilde H.
		\]
		Next,
		\[
		\widetilde X \widetilde H \widetilde X
		=
		\left[
		\begin{array}{cc}
			J^{-1} & 0\\
			0 & 0
		\end{array}
		\right]
		\left[
		\begin{array}{cc}
			J & 0\\
			0 & 0
		\end{array}
		\right]
		\left[
		\begin{array}{cc}
			J^{-1} & 0\\
			0 & 0
		\end{array}
		\right]
		=
		\left[
		\begin{array}{cc}
			J^{-1} & 0\\
			0 & 0
		\end{array}
		\right]
		=
		\widetilde X.
		\]
		Finally, since \(\ind(H)=1\), we must check
		\(
		\widetilde H^{\,2}\widetilde X=\widetilde H.
		\)
		Indeed,
		\[
		\widetilde H^{\,2}\widetilde X
		=
		\left[
		\begin{array}{cc}
			J^2 & 0\\
			0 & 0
		\end{array}
		\right]
		\left[
		\begin{array}{cc}
			J^{-1} & 0\\
			0 & 0
		\end{array}
		\right]
		=
		\left[
		\begin{array}{cc}
			J & 0\\
			0 & 0
		\end{array}
		\right]
		=
		\widetilde H.
		\]
		Thus \(\widetilde X\) is the Drazin inverse of \(\widetilde H\). By
		similarity invariance of the Drazin inverse, it follows that
		\[
		H^D
		=
		S\widetilde X S^{-1}
		=
		S
		\left[
		\begin{array}{cc}
			J^{-1} & 0\\
			0 & 0
		\end{array}
		\right]
		S^{-1}.
		\]
		
		For the projector \(P_0\), we use the definition
		\(
		P_0=I-HH^D.
		\)
		Since
		\[
		HH^D
		=
		S\widetilde H S^{-1}
		S\widetilde X S^{-1}
		=
		S(\widetilde H\widetilde X)S^{-1}
		=
		S
		\left[
		\begin{array}{cc}
			I_q & 0\\
			0 & 0
		\end{array}
		\right]
		S^{-1},
		\]
		we obtain
		\[
		P_0
		=
		I-
		S
		\left[
		\begin{array}{cc}
			I_q & 0\\
			0 & 0
		\end{array}
		\right]
		S^{-1}
		=
		S
		\left[
		\begin{array}{cc}
			0 & 0\\
			0 & I_\nu
		\end{array}
		\right]
		S^{-1}.
		\]
		This proves the result.
	\end{pf}
	
	\begin{lem}\label{lem:projected_columns}
		Let $H \in {\mathbb R}^{n \times n}$ satisfy $\ind(H)=1$ and let
		$P_0=I-HH^D$.
		If $U,V \in {\mathbb R}^{n \times r}$ satisfy
		\(
		P_0U=0,
		\
		V^TP_0=0,
		\)
		then for the decomposition induced by Lemma \ref{lem:block_drazin} one has
		\[
		S^{-1}U=
		\left[
		\begin{array}{c}
			U_1\\
			0
		\end{array}
		\right],
		\
		S^TV=
		\left[
		\begin{array}{c}
			V_1\\
			0
		\end{array}
		\right]
		\]
		for some matrices $U_1,V_1 \in {\mathbb R}^{q \times r}$.
	\end{lem}
	
	\begin{pf}
		From
		\(
		P_0U=0
		\)
		and Lemma \ref{lem:block_drazin} we obtain
		\[
		\left[
		\begin{array}{cc}
			0 & 0\\
			0 & I_\nu
		\end{array}
		\right]
		S^{-1}U=0,
		\]
		hence the lower block of $S^{-1}U$ must vanish.
		Similarly, from
		\(
		V^TP_0=0
		\)
		we obtain
		\[
		V^TS
		\left[
		\begin{array}{cc}
			0 & 0\\
			0 & I_\nu
		\end{array}
		\right]
		=0,
		\]
		which implies that the lower block of $S^TV$ vanishes.
	\end{pf}
	The preceding technical lemmas provide the structural decomposition
	of matrices with semisimple zero eigenvalue and the corresponding
	representation of the Drazin inverse and spectral projector. We are now
	in a position to derive a determinant identity that extends the classical
	matrix determinant lemma to the singular setting. In particular, the
	following result shows that, under natural compatibility conditions,
	the pseudodeterminant of a finite-rank perturbation admits a multiplicative
	representation involving the Drazin inverse, thereby capturing the effect
	of the perturbation on the nonzero spectrum.
	\begin{thm}\label{thm:pdet_drazin}
		Let $H \in {\mathbb R}^{n \times n}$ satisfy $\ind(H)=1$, and let
		$U,V \in {\mathbb R}^{n \times r}$ satisfy
		\(
		P_0U=0,
		\
		V^TP_0=0,
		\)
		where $P_0=I-HH^D$.
		Then
		\[
		\pdet(H+UV^T)
		=
		\pdet(H)\det(I_r+V^TH^DU).
		\]
	\end{thm}
	
	\begin{pf}
		By Lemma \ref{lem:block_drazin} there exists a nonsingular matrix $S$ such that
		\[
		S^{-1}HS=
		\left[
		\begin{array}{cc}
			J & 0\\
			0 & 0
		\end{array}
		\right],
		\
		H^D=
		S
		\left[
		\begin{array}{cc}
			J^{-1} & 0\\
			0 & 0
		\end{array}
		\right]
		S^{-1}.
		\]
		By Lemma \ref{lem:projected_columns},
		\[
		S^{-1}U=
		\left[
		\begin{array}{c}
			U_1\\
			0
		\end{array}
		\right],
		\
		S^TV=
		\left[
		\begin{array}{c}
			V_1\\
			0
		\end{array}
		\right]
		\]
		for some $U_1,V_1 \in {\mathbb R}^{q \times r}$.
		Hence
		\[
		S^{-1}(H+UV^T)S=
		\left[
		\begin{array}{cc}
			J+U_1V_1^T & 0\\
			0 & 0
		\end{array}
		\right].
		\]
		Therefore,
		\(
		\pdet(H+UV^T)=\det(J+U_1V_1^T).
		\)
		Applying the classical matrix determinant lemma to the nonsingular matrix $J$,
		we obtain
		\[
		\det(J+U_1V_1^T)
		=
		\det(J)\det(I_r+V_1^TJ^{-1}U_1).
		\]
		Now
		\[
		V^TH^DU
		=
		V^TS
		\left[
		\begin{array}{cc}
			J^{-1} & 0\\
			0 & 0
		\end{array}
		\right]
		S^{-1}U
		=
		V_1^TJ^{-1}U_1.
		\]
		Combining the above equalities gives
		\[
		\pdet(H+UV^T)
		=
		\pdet(H)\det(I_r+V^TH^DU).
		\]
		This completes the proof.
	\end{pf}
	
	Theorem \ref{thm:pdet_drazin} shows that, under compatibility conditions
	with the eigenprojector $P_0$, the singular case reduces exactly to the
	nonsingular case on the invariant subspace corresponding to the nonzero
	spectrum of $H$.
	
	\begin{rem}\rm
		The assumption $\ind(H)=1$ ensures that the eigenvalue zero is
		semisimple and allows a direct decomposition of $H$ into a
		nonsingular block and a zero block. This structure is crucial
		for reducing the problem to the nonsingular case.
		
		If $\ind(H)>1$, the situation becomes more involved. In this case,
		there exists a decomposition of the form
		\[
		S^{-1}HS =
		\left[
		\begin{array}{cc}
			J & 0\\
			0 & N
		\end{array}
		\right],
		\]
		where $J$ is nonsingular and $N$ is a nilpotent matrix with
		$N^k=0$ for some $k>1$.
		
		Under a finite-rank perturbation, the matrix $H + UV^T$ induces
		a coupling between the nonsingular and nilpotent parts. As a
		result, the determinant no longer depends solely on the block $J$,
		and additional terms involving the nilpotent structure appear.
		In particular, a direct identity of the form
		\(
		\pdet(H + UV^T)
		=
		\pdet(H)\det(I + V^T H^D U)
		\)
		does not hold in general without further structural assumptions.
		Instead, one obtains expansions that involve higher-order terms
		related to the nilpotent block.
		A possible approach is to consider regularized expressions of the form
		\(
		\det(H + \varepsilon I + UV^T),
		\)
		and analyze their asymptotic behavior as $\varepsilon \to 0^+$. In this setting,
		the leading-order term recovers the pseudodeterminant, while higher-order
		terms encode the influence of the nilpotent structure.
		A detailed treatment of the case $\ind(H)>1$ would require a refined
		analysis of Jordan chains and higher-order asymptotic contributions.
		In this setting, the zero eigenvalue is associated with nontrivial
		nilpotent blocks, so that the decomposition of $H$ involves terms of the form
		\[
		\left[
		\begin{array}{cc}
			J & 0\\
			0 & N
		\end{array}
		\right],
		\]
		where $N$ is nilpotent but nonzero. As a consequence, the inverse-like
		structure captured by the Drazin inverse is no longer reduced to a simple
		block inversion of $J$, and interactions between the nilpotent part and
		the perturbation $UV^T$ must be taken into account.
		Moreover, the regularized determinant $\det(H+\varepsilon I+UV^T)$
		exhibits a more intricate asymptotic behavior, involving higher-order
		terms in $\varepsilon$ beyond the leading contribution associated with
		the pseudodeterminant. These effects reflect the presence of Jordan chains
		and the coupling between generalized eigenvectors.
		A systematic treatment of this case is therefore substantially more
		involved and is left for future work.
	\end{rem}
	
	Motivated by the above discussion, we now derive a regularized
	determinant representation which remains valid in the singular case
	$\ind(H)=1$ and captures the contribution of the nonzero spectrum.
	
	\begin{thm}\label{thm:regularized_det}
		Let $H \in {\mathbb R}^{n \times n}$ satisfy $\ind(H)=1$, and let the algebraic
		multiplicity of the zero eigenvalue be equal to $\nu$.
		Let $U,V \in {\mathbb R}^{n \times r}$ satisfy
		\(
		P_0U=0,
		\
		V^TP_0=0.
		\)
		Then
		\[
		\lim_{\varepsilon \to 0} \varepsilon^{-\nu}\det(H+\varepsilon I+UV^T)
		=
		\pdet(H)\det(I_r+V^TH^DU).
		\]
	\end{thm}
	
	\begin{pf}
		Using the same decomposition as in the proof of
		Theorem \ref{thm:pdet_drazin}, we obtain
		\[
		S^{-1}(H+\varepsilon I+UV^T)S=
		\left[
		\begin{array}{cc}
			J+\varepsilon I_q+U_1V_1^T & 0\\
			0 & \varepsilon I_\nu
		\end{array}
		\right].
		\]

		Hence
		\[
		\det(H+\varepsilon I+UV^T)
		=
		\varepsilon^\nu\det(J+\varepsilon I_q+U_1V_1^T).
		\]
		Therefore,
		\[
		\varepsilon^{-\nu}\det(H+\varepsilon I+UV^T)
		=
		\det(J+\varepsilon I_q+U_1V_1^T).
		\]
		Passing to the limit as $\varepsilon \to 0$ gives
		\[
		\lim_{\varepsilon \to 0} \varepsilon^{-\nu}\det(H+\varepsilon I+UV^T)
		=
		\det(J+U_1V_1^T).
		\]
		
		Since $J$ is nonsingular, the matrix determinant lemma yields
		\[
		\det(J+U_1V_1^T)
		=
		\det(J)\det(I_r+V_1^T J^{-1}U_1).
		\]
		By the definition of the pseudodeterminant, we have
		\(
		\det(J)=\pdet(H),
		\)
		and by the block representation of the Drazin inverse,
		\[
		H^D = S
		\left[
		\begin{array}{cc}
			J^{-1} & 0\\
			0 & 0
		\end{array}
		\right]
		S^{-1},
		\]
		together with the assumptions $P_0U=0$ and $V^TP_0=0$, it follows that
		\(
		V_1^T J^{-1} U_1 = V^T H^D U.
		\)
		Therefore,
		\[
		\det(J+U_1V_1^T)
		=
		\pdet(H)\det(I_r+V^T H^D U),
		\]
		which proves the result.
	\end{pf}
	
	Theorem \ref{thm:regularized_det} provides a regularized determinant
	representation for singular matrices. It can be interpreted as a
	regularized counterpart of Theorem~\ref{thm:pdet_drazin}, which gives
	a direct formula for the pseudodeterminant under compatibility
	conditions with the nullspace. In contrast, Theorem
	\ref{thm:regularized_det} derives the same structure through a limiting
	argument, allowing one to work within a nonsingular framework.
	This formulation is particularly useful for spectral analysis, as it
	separates the contribution of the nonzero spectrum from the singular
	part and enables stable limiting arguments.

	\section{Spectral implications of low-rank perturbations}\label{sec:Spectral implications of low-rank perturbations}
	
	In this section we study the effect of finite-rank perturbations
	on the characteristic polynomial and on the spectrum of a matrix.

	\begin{thm}\label{thm:charpoly_decomposition}
		Let $A \in {\mathbb R}^{n \times n}$ and let $U,V \in {\mathbb R}^{n \times r}$.
		Then
		\[
		\det(\lambda I - A - UV^T)
		=
		\det(\lambda I - A)
		-
		\sum_{i=1}^r
		v_i^T \adj(\lambda I - A - \Delta_{i-1}) u_i.
		\]
	\end{thm}
	
	\begin{pf}
		Apply Theorem \ref{thm:main} to the matrix
		\(
		H=\lambda I-A
		\)
		and to the rank-one updates
		\(
		(-u_i)v_i^T, \ i=1,\dots,r.
		\)
		Then
		\[
		\det(\lambda I-A-UV^T)
		=
		\det(\lambda I-A)
		+
		\sum_{i=1}^r
		v_i^T \adj\left((\lambda I-A)-\Delta_{i-1}\right)(-u_i),
		\]
		which yields
		\[
		\det(\lambda I-A-UV^T)
		=
		\det(\lambda I-A)
		-
		\sum_{i=1}^r
		v_i^T \adj(\lambda I-A-\Delta_{i-1})u_i.
		\]
	\end{pf}

	Theorem \ref{thm:charpoly_decomposition} provides an explicit
	representation of the characteristic polynomial under finite-rank
	perturbations.

	\begin{thm}\label{thm:eigenvalue_condition}
		Let $\lambda$ be such that $\lambda I - A - \Delta_{i-1}$ is nonsingular.
		Then $\lambda$ is an eigenvalue of $A + \Delta_i$ if and only if
		\(
		1 - v_i^T (\lambda I - A - \Delta_{i-1})^{-1} u_i = 0.
		\)
	\end{thm}
	
	\begin{pf}
		Since
		\(
		A + \Delta_i = A + \Delta_{i-1} + u_i v_i^T,
		\)
		we have
		\(
		\lambda I - A - \Delta_i
		=
		(\lambda I - A - \Delta_{i-1}) - u_i v_i^T.
		\)
		Let
		\(
		M = \lambda I - A - \Delta_{i-1}.
		\)
		Then $M$ is nonsingular by assumption, and
		\(
		\det(\lambda I - A - \Delta_i)
		=
		\det(M - u_i v_i^T).
		\)
		By Theorem \ref{thm:rank1}, applied to the matrix $M$ and the
		rank-one perturbation $-u_i v_i^T$, we obtain
		\(
		\det(M - u_i v_i^T)
		=
		\det(M)
		\left(
		1 - v_i^T M^{-1} u_i
		\right).
		\)
		Hence
		\[
		\det(\lambda I - A - \Delta_i)
		=
		\det(\lambda I - A - \Delta_{i-1})
		\left(
		1 - v_i^T (\lambda I - A - \Delta_{i-1})^{-1} u_i
		\right).
		\]
		Therefore $\lambda$ is an eigenvalue of $A + \Delta_i$ if and only if
		\(
		1 - v_i^T (\lambda I - A - \Delta_{i-1})^{-1} u_i = 0.
		\)
	\end{pf}
	Theorem \ref{thm:eigenvalue_condition} shows that each rank-one
	perturbation introduces a scalar nonlinear equation governing the
	shift of eigenvalues. This condition is closely related to the
	resolvent of the perturbed matrix, which admits an explicit
	representation under rank-one updates.

	\begin{thm}\label{thm:stability_shift}
		Assume that all eigenvalues of $A$ lie in the open left half-plane.
		Let $u,v \in {\mathbb R}^n$. If
		\(
		1 - v^T (\lambda I - A)^{-1} u \neq 0
		\)
		for all $\lambda$ with $\Re(\lambda)\ge 0$, then all eigenvalues of
		$A + uv^T$ also lie in the open left half-plane.
	\end{thm}
	
	\begin{pf}
		Since all eigenvalues of $A$ lie in the open left half-plane, the matrix
		\(
		\lambda I-A
		\)
		is nonsingular for every $\lambda$ with $\Re(\lambda)\ge 0$.
		Let $\lambda$ be an eigenvalue of $A+uv^T$. Then
		\(
		\det(\lambda I-A-uv^T)=0.
		\)
		By Theorem \ref{thm:eigenvalue_condition}, this implies
		\(
		1-v^T(\lambda I-A)^{-1}u=0.
		\)
		The assumption of the theorem excludes this possibility for all
		$\lambda$ with $\Re(\lambda)\ge 0$. Therefore $A+uv^T$ has no
		eigenvalues in the closed right half-plane, and hence all its
		eigenvalues lie in the open left half-plane.
	\end{pf}
	Theorem \ref{thm:stability_shift} provides a frequency-domain
	condition for preservation of stability under rank-one perturbations.

	\section{Applications to estimation and control}\label{sec:Applications to estimation and control}
	
	In this section we demonstrate how the developed determinant
	identities can be applied in estimation and control problems.
	In particular, we focus on covariance updates and quantities
	based on the log-determinant, which play a central role in
	information-theoretic formulations of control and estimation.
	The derived representations reveal a common structural principle:
	determinant and log-determinant quantities evolve through successive
	contributions associated with individual directions. This viewpoint
	provides a unified interpretation across several areas of control.
	In estimation, it explains how information is accumulated through
	measurements. In recursive filtering, it quantifies the reduction
	of uncertainty. In reachability analysis, it characterizes the
	expansion of the reachable set.
	
	The section is organized as follows. We begin with covariance updates
	and their interpretation in terms of information accumulation.
	We then connect these results to Kalman-type filtering, where
	low-rank updates arise naturally in the information form.
	Next, we present a simple singular example that isolates the role
	of the pseudodeterminant and the Drazin inverse in a combined
	control and information setting. Finally, we turn to controllability
	Gramians, where the developed framework provides a geometric and
	quantitative description of the evolution of the reachable set.
	
	\subsection{Covariance updates and information accumulation}
	Consider a symmetric positive definite matrix
	\(
	P \in {\mathbb R}^{n \times n}
	\)
	and a sequence of rank-one updates of the form
	\[
	P_k = P + \Delta_k,
	\qquad
	\Delta_k = \sum_{i=1}^k u_i u_i^T,
	\qquad
	\Delta_0 = 0.
	\]
	In particular, $P_0 = P$.
	Such updates naturally arise in recursive estimation and Kalman
	filtering, where the covariance matrix is updated as new measurements
	are incorporated. They also appear in optimal experiment design and
	sensor placement, where the matrix $P_k$ represents the accumulated
	information from individual measurements.
	Applying the determinant identity derived in the previous section,
	we obtain
	\[
	\det(P_k)
	=
	\det(P)
	\prod_{i=1}^k
	\left(
	1 + u_i^T (P + \Delta_{i-1})^{-1} u_i
	\right),
	\]
	and consequently
	\[
	\log \det(P_k)
	=
	\log \det(P)
	+
	\sum_{i=1}^k
	\log
	\left(
	1 + u_i^T (P + \Delta_{i-1})^{-1} u_i
	\right).
	\]
	
	This representation shows that the log-determinant evolves additively
	under successive rank-one updates, with each term
	\(
	\log
	\left(
	1 + u_i^T (P + \Delta_{i-1})^{-1} u_i
	\right)
	\)
	capturing the incremental contribution of the direction $u_i$.
	From an estimation and control perspective, the quantity
	$\log \det(P_k)$ is widely used as a measure of uncertainty or
	information content. The above formula makes this interpretation
	explicit: each update contributes according to the quadratic form
	\(
	u_i^T (P + \Delta_{i-1})^{-1} u_i,
	\)
	which reflects how the new direction $u_i$ interacts with the
	current covariance structure.
	If $u_i$ lies in directions that are already well represented in
	$P + \Delta_{i-1}$, the corresponding contribution is small,
	whereas directions that are poorly represented yield significantly
	larger contributions. A detailed analysis of this mechanism is
	given in Example~\ref{ex:contribution}.
	This provides a precise analytical explanation of a phenomenon
	commonly observed in practice: redundant measurements yield
	diminishing returns, while measurements aligned with previously
	unobserved directions lead to substantial information gain.
	The derived identities therefore offer a transparent tool for
	analyzing and designing update strategies in estimation and
	control problems.
	
	The above representation not only provides an exact decomposition,
	but also enables a quantitative analysis of the rate at which the
	log-determinant evolves under successive updates. This can be made
	precise through the following bounds.
	
	\begin{thm}\label{thm:cov_bounds}
		Let $P \in {\mathbb R}^{n \times n}$ be symmetric positive definite and define
		\[
		P_0 = P,
		\qquad
		P_k = P + \sum_{i=1}^k u_i u_i^T,
		\quad k \ge 1.
		\]
		Then
		\[
		\log \det(P_k) - \log \det(P)
		=
		\sum_{i=1}^k
		\log\left(1 + u_i^T P_{i-1}^{-1} u_i\right).
		\]
		Moreover, the following bounds hold:
		\[
		\sum_{i=1}^k
		\frac{u_i^T P_{i-1}^{-1} u_i}{1 + u_i^T P_{i-1}^{-1} u_i}
		\;\le\;
		\log \frac{\det(P_k)}{\det(P)}
		\;\le\;
		\sum_{i=1}^k
		u_i^T P_{i-1}^{-1} u_i.
		\]
	\end{thm}
	
	\begin{pf}
		The identity follows directly from Theorem~\ref{thm:logdet}.
		
		To establish the bounds, observe that for all $x > -1$,
		\(
		\frac{x}{1+x} \le \log(1+x) \le x.
		\)
		Applying these inequalities with
		\(
		x = u_i^T P_{i-1}^{-1} u_i \ge 0
		\)
		and summing over $i=1,\dots,k$ yields the result.
	\end{pf}

	\subsection{Connection to Kalman-type updates}
	
	The covariance update described above admits a natural interpretation
	in the context of recursive estimation, in particular in the information
	form of the Kalman filter. In this formulation, the inverse covariance
	matrix evolves through additive low-rank updates, which makes it directly
	amenable to the determinant identities developed in this paper.
	
	Let $P \in {\mathbb R}^{n \times n}$ be a symmetric positive definite
	covariance matrix. In many estimation problems, it is convenient to work
	with the inverse covariance, which is updated as
	\(
	P^{-1} \to P^{-1} + C^T R^{-1} C,
	\)
	where $C$ denotes the measurement matrix and $R$ the noise covariance.
	In the case of sequential or scalar measurements, this leads to updates
	of the form
	\(
	P_k^{-1} = P^{-1} + \sum_{i=1}^k v_i v_i^T.
	\)
	Since $P^{-1} \succ 0$ and $v_i v_i^T \succeq 0$, it follows that
	$P_k^{-1} \succ 0$ and hence $P_k \succ 0$ for all $k \ge 0$.
	Applying the determinant identity for rank-one updates to $P^{-1}$,
	we obtain
	\[
	\det(P_k)
	=
	\det(P)
	\prod_{i=1}^k
	\frac{1}{1 + v_i^T P_{i-1} v_i},
	\]
	where $P_{i-1}$ denotes the covariance after $(i-1)$ updates.
	
	This representation makes explicit how uncertainty evolves as new
	measurements are incorporated. Each factor
	\(
	\frac{1}{1 + v_i^T P_{i-1} v_i}
	\)
	is strictly less than one whenever $v_i \neq 0$ and quantifies the
	reduction of the covariance volume due to the $i$-th measurement.
	Moreover, the above identity does not only provide a multiplicative
	representation, but also enables a quantitative analysis of the rate
	at which the covariance contracts under successive measurements.
	
	\begin{thm}\label{thm:kalman_monotone}
		Let $P \in {\mathbb R}^{n \times n}$ be symmetric positive definite and define
		\[
		P_0 = P,
		\qquad
		P_k^{-1} = P^{-1} + \sum_{i=1}^k v_i v_i^T,
		\quad k \ge 1.
		\]
		Then $P_k$ is symmetric positive definite for all $k \ge 0$, and
		\[
		\log \det(P_k)
		=
		\log \det(P)
		-
		\sum_{i=1}^k
		\log\left(1 + v_i^T P_{i-1} v_i\right).
		\]
		Moreover, the sequence $\det(P_k)$ is nonincreasing, and it is
		strictly decreasing whenever $v_i \neq 0$ for all $i=1,\dots,k$.
		If, in addition, there exists $\beta > 0$ such that
		\[
		v_i^T P_{i-1} v_i \ge \beta,
		\quad i=1,\dots,k,
		\] 
		then
		\[
		\det(P_k) \le \det(P)\,(1+\beta)^{-k}.
		\]
	\end{thm}
	
	\begin{pf}
		Since $P$ is symmetric positive definite, so is $P^{-1}$. Each matrix
		$v_i v_i^T$ is symmetric positive semidefinite, and therefore
		\(
		P_k^{-1} = P^{-1} + \sum_{i=1}^k v_i v_i^T
		\)
		is symmetric positive definite for every $k \ge 1$. Consequently, $P_k$
		is symmetric positive definite for all $k \ge 0$.
		Applying the determinant identity to $P^{-1}$, we obtain
		\[
		\det(P_k^{-1})
		=
		\det(P^{-1})
		\prod_{i=1}^k
		\left(
		1 + v_i^T P_{i-1} v_i
		\right).
		\]
		Since
		\(
		\det(P_k) = \det(P_k^{-1})^{-1},
		\)
		it follows that
		\[
		\det(P_k)
		=
		\det(P)
		\prod_{i=1}^k
		\frac{1}{1 + v_i^T P_{i-1} v_i}.
		\]
		Taking logarithms yields
		\[
		\log \det(P_k)
		=
		\log \det(P)
		-
		\sum_{i=1}^k
		\log\left(1 + v_i^T P_{i-1} v_i\right).
		\]
		
		Since $P_{i-1}$ is positive definite, we have
		\(
		v_i^T P_{i-1} v_i \ge 0
		\)
		for every $i$. Hence each factor
		\(
		\frac{1}{1 + v_i^T P_{i-1} v_i}
		\le 1,
		\)
		which shows that $\det(P_k)$ is nonincreasing. If $v_i \neq 0$, then
		\(
		v_i^T P_{i-1} v_i > 0,
		\)
		and therefore
		\(
		\frac{1}{1 + v_i^T P_{i-1} v_i} < 1.
		\)
		In this case, $\det(P_k)$ is strictly decreasing.
		If $v_i^T P_{i-1} v_i \ge \beta > 0$ for all $i=1,\dots,k$, then
		\(
		\frac{1}{1 + v_i^T P_{i-1} v_i}
		\le
		\frac{1}{1+\beta},
		\)
		and therefore
		\[
		\det(P_k)
		=
		\det(P)
		\prod_{i=1}^k
		\frac{1}{1 + v_i^T P_{i-1} v_i}
		\le
		\det(P)\,(1+\beta)^{-k},
		\]
		which proves the result.
	\end{pf}
	
	The above result admits a natural interpretation in information-theoretic
	terms. The quantity
	\(
	\log \det(P_k)
	\)
	is commonly used as a measure of uncertainty, and the representation
	derived above shows that each measurement contributes an additive term
	\(
	-\log\left(1 + v_i^T P_{i-1} v_i\right),
	\)
	which can be interpreted as the information gain associated with the
	$i$-th update.
	The magnitude of this contribution depends on the interaction between
	the measurement direction $v_i$ and the current covariance $P_{i-1}$.
	If $v_i$ corresponds to a direction with high uncertainty, then
	$v_i^T P_{i-1} v_i$ is large and the resulting reduction in
	$\det(P_k)$ is significant. Conversely, if the measurement is aligned
	with directions that are already well estimated, the contribution is
	small.
	
	This provides a precise analytical explanation of a fundamental
	property of recursive estimation: measurements that probe previously
	uncertain directions are the most informative, while redundant
	measurements yield diminishing returns. The derived identities
	therefore offer a transparent tool for analyzing and designing
	measurement strategies in Kalman-type filtering and related
	estimation problems.
	
	\subsection{A singular $3 \times 3$ example: control and information perspective}
	
	The covariance and Kalman-type updates considered above illustrate the
	nonsingular information-theoretic side of the proposed framework. We now
	turn to a simple singular example, which makes it possible to isolate
	the role of the pseudodeterminant and the Drazin inverse in a transparent
	control and information setting. In particular, this example shows how
	low-rank perturbations affect both the spectrum and the effective
	information content of the system.
	
	Consider the singular matrix
	\[
	A =
	\left[
	\begin{array}{ccc}
		-1 & 0 & 0\\
		0 & -2 & 0\\
		0 & 0 & 0
	\end{array}
	\right].
	\]
	This matrix can be interpreted as encoding a system with two
	informative directions and one direction with zero information
	(e.g., infinite uncertainty in a covariance interpretation).
	Accordingly, its determinant is zero, while its pseudodeterminant
	captures the nondegenerate part,
	\(
	\pdet(A)=(-1)(-2)=2.
	\)
	Since the zero eigenvalue is semisimple, the Drazin inverse is
	\[
	A^D =
	\left[
	\begin{array}{ccc}
		-1 & 0 & 0\\
		0 & -\frac{1}{2} & 0\\
		0 & 0 & 0
	\end{array}
	\right],
	\]
	and the spectral projector onto the nullspace is
	\[
	P_0 = I-AA^D =
	\left[
	\begin{array}{ccc}
		0 & 0 & 0\\
		0 & 0 & 0\\
		0 & 0 & 1
	\end{array}
	\right].
	\]

	\textit{Update on the informative subspace.}
	We first consider a rank-one update acting only on the
	informative directions. Let
	\[
	u_s =
	\left[
	\begin{array}{c}
		1\\
		0\\
		0
	\end{array}
	\right],
	\
	v_s =
	\left[
	\begin{array}{c}
		p\\
		q\\
		0
	\end{array}
	\right],
	\]
	and define
	\(
	A_s = A + u_s v_s^T.
	\)
	Since the update does not affect the nullspace, the assumptions
	of the pseudodeterminant formula are satisfied, and
	\(
	\pdet(A_s)
	=
	\pdet(A)\left(1+v_s^T A^D u_s\right).
	\)
	A direct computation yields
	\(
	v_s^T A^D u_s = -p,
	\
	\pdet(A_s)=2(1-p).
	\)
	
	This shows that the update modifies the information content
	within the already informative subspace, while the overall
	determinant remains zero due to the unchanged null direction.
	Thus, the pseudodeterminant isolates the effective information
	variation that would be invisible to the standard determinant.

	\textit{Update on the nullspace.}
	We now consider a rank-one update acting on the previously
	uninformative direction. Let
	\[
	u_0 =
	\left[
	\begin{array}{c}
		0\\
		0\\
		1
	\end{array}
	\right],
	\
	v_0 =
	\left[
	\begin{array}{c}
		a\\
		b\\
		c
	\end{array}
	\right],
	\]
	and define
	\(
	A_0 = A + u_0 v_0^T.
	\)
	By Theorem \ref{thm:rank1},
	\(
	\det(A_0)=2c.
	\)
	Hence, as soon as $c \neq 0$, the matrix becomes nonsingular.
	From an information-theoretic viewpoint, this corresponds to
	injecting information into a previously unobservable or
	uncertain direction.
	The spectral interpretation is explicit: the new eigenvalue is $c$,
	so the formerly neutral mode is shifted and becomes informative.
	This transition is also captured by the resolvent-based
	characterization of eigenvalues.

	\textit{Sequential updates and information accumulation.}
	Finally, consider the successive perturbation
	\(
	A_{2}=A+u_0 v_0^T + u_s v_s^T.
	\)
	Then
	\(
	\det(A_{2})=2c(1-p),
	\)
	which is recovered exactly by the determinant dynamics developed
	in this paper.
	This decomposition shows that each rank-one update contributes
	multiplicatively to the total determinant, or additively at the
	logarithmic level. In particular, the update $u_0 v_0^T$ activates
	a previously degenerate direction, while $u_s v_s^T$ modifies the
	information content within the already active subspace.

	Overall, this example illustrates how the proposed framework
	provides a transparent decomposition of structural and
	information-theoretic changes under low-rank perturbations.
	In particular, it distinguishes between
	\emph{activation of new directions} and
	\emph{refinement of existing ones}, both of which contribute
	to the evolution of determinant-based quantities. 
	
	\subsection{Gramian determinants and reachability}
	
	Consider the linear discrete-time system
	\begin{equation}
		x_{t+1} = A x_t + B u_t,
		\ x_0 = 0,
		\label{eq:linear_discrete}
	\end{equation}
	where $A \in {\mathbb R}^{n \times n}$ and $B \in {\mathbb R}^{n \times m}$.
	The state reached after $N$ steps is given by
	\[
	x_N = \sum_{i=0}^{N-1} A^{N-1-i} B u_i,
	\]
	which shows that each input $u_i$ is propagated through the system
	over $N-1-i$ steps.
	The set of all states reachable in $N$ steps coincides with
	$\operatorname{Im}(W_N)$, where
	\[
	W_N=\sum_{i=0}^{N-1}A^iBB^T(A^T)^i\in \mathbb{R}^{n\times n}
	\]
	is the finite-horizon controllability Gramian. Moreover, under a unit
	input-energy constraint, the reachable set is the ellipsoid
	\(
	\{x\in \operatorname{Im}(W_N): x^T W_N^D x\le 1\},
	\)
	which reduces to
	\(
	\{x\in \mathbb R^n: x^T W_N^{-1}x\le 1\}
	\)
	when $W_N$ is nonsingular.
	This matrix characterizes how input energy is mapped into the state
	space and quantifies the anisotropic distribution of reachable
	directions. In particular, $W_N$ is nonsingular if and only if the
	system is reachable over $N$ steps.
	The Gramian admits a natural decomposition into low-rank contributions.
	Writing $B = [b_1,\dots,b_m]$ and defining
	\(
	u_{(i,j)} = A^i b_j,
	\)
	we obtain
	\(
	W_N
	=
	\sum_{i=0}^{N-1} \sum_{j=1}^m
	u_{(i,j)} u_{(i,j)}^T.
	\)
	Reindexing these vectors as $\{u_\ell\}$ yields a representation of the form
	\(
	W_N = \sum_{\ell=1}^{Nm} u_\ell u_\ell^T,
	\)
	which allows the application of the determinant identities developed earlier.
	In order to apply multiplicative determinant formulas in a uniform way,
	including the singular case, we introduce a regularized sequence of
	partial sums
	\(
	\widetilde{W}_0 = \varepsilon I,
	\
	\widetilde{W}_\ell
	=
	\varepsilon I + \sum_{j=1}^{\ell} u_j u_j^T,
	\ \ell=1,\dots,Nm,
	\)
	so that $\widetilde{W}_{Nm} = \varepsilon I + W_N$.
	Each matrix $\widetilde{W}_\ell$ is symmetric positive definite.
	Applying the determinant identity for rank-one updates, we obtain
	\[
	\det(\widetilde{W}_{Nm})
	=
	\det(\widetilde{W}_0)
	\prod_{\ell=1}^{Nm}
	\left(
	1 + u_\ell^T \widetilde{W}_{\ell-1}^{-1} u_\ell
	\right).
	\]
	Since $\det(\widetilde{W}_0)=\det(\varepsilon I)=\varepsilon^n$, it follows that
	\[
	\det(\widetilde{W}_{Nm})
	=
	\varepsilon^n
	\prod_{\ell=1}^{Nm}
	\left(
	1 + u_\ell^T \widetilde{W}_{\ell-1}^{-1} u_\ell
	\right).
	\]
	
	Passing to the limit as $\varepsilon \to 0$ requires appropriate
	normalization. Since $W_N$ is symmetric positive semidefinite, let
	$\lambda_1,\dots,\lambda_r > 0$ denote its nonzero eigenvalues, where
	$r = \rank(W_N) \le n$. Then
	\(
	\det(\varepsilon I + W_N)
	=
	\varepsilon^{\,n-r} \prod_{i=1}^r (\lambda_i + \varepsilon).
	\)
	Dividing by $\varepsilon^{n-r}$ and passing to the limit yields
	\[
	\pdet(W_N)
	=
	\prod_{i=1}^r \lambda_i
	=
	\lim_{\varepsilon \to 0}
	\varepsilon^{-(n-r)} \det(\varepsilon I + W_N).
	\]
	
	This representation shows that the pseudodeterminant can be interpreted
	as the product of incremental contributions associated with directions
	that expand the reachable subspace. In particular, only those vectors
	$u_\ell$ that introduce new independent directions contribute to the
	leading-order term in the limit, whereas directions already contained
	in the current span affect only higher-order terms in $\varepsilon$.
	Moreover, this formulation provides a quantitative description of how
	reachability evolves as new directions are added. To avoid degeneracies associated with the initial singularity of
	$W_0 = 0$ and to obtain a representation valid in both singular
	and nonsingular settings, we introduce a regularized formulation.
	The pseudodeterminant then arises naturally as the limit of the
	regularized determinant and captures the intrinsic volume of the
	reachable set independently of rank deficiencies.
	
	\begin{thm}\label{thm:gramian_growth_pdet}
		Let
		\(
		W_N = \sum_{\ell=1}^{Nm} u_\ell u_\ell^T,
		\)
		and, for $\varepsilon > 0$, define the regularized partial sums
		\[
		\widetilde{W}_0(\varepsilon)=\varepsilon I,
		\
		\widetilde{W}_\ell(\varepsilon)
		=
		\varepsilon I + \sum_{j=1}^{\ell} u_j u_j^T,
		\
		\ell=1,\dots,Nm.
		\]
		Then
		\begin{equation}
			\det\bigl(\widetilde{W}_{Nm}(\varepsilon)\bigr)
			=
			\varepsilon^n
			\prod_{\ell=1}^{Nm}
			\left(
			1 + u_\ell^T \widetilde{W}_{\ell-1}(\varepsilon)^{-1} u_\ell
			\right).
			\label{eq:gramian_regularized_product}
		\end{equation}
		
		If \(r=\rank(W_N)\), then
		\begin{equation}
			\pdet(W_N)
			=
			\lim_{\varepsilon\to 0}
			\varepsilon^{-(n-r)}
			\det\bigl(\widetilde{W}_{Nm}(\varepsilon)\bigr).
			\label{eq:gramian_pdet_limit}
		\end{equation}
		Equivalently,
		\begin{equation}
			\pdet(W_N)
			=
			\lim_{\varepsilon\to 0}
			\varepsilon^{r}
			\prod_{\ell=1}^{Nm}
			\left(
			1 + u_\ell^T \widetilde{W}_{\ell-1}(\varepsilon)^{-1} u_\ell
			\right).
			\label{eq:gramian_pdet_product_limit}
		\end{equation}
		
		Moreover, if \(\pdet(W_N)>0\), then
		\[
		\log \pdet(W_N)
		=
		\lim_{\varepsilon\to 0}
		\left[
		\log \det\bigl(\widetilde{W}_{Nm}(\varepsilon)\bigr)
		-
		(n-r)\log \varepsilon
		\right].
		\]
	\end{thm}
	
	\begin{pf}
		For each fixed \(\varepsilon>0\), the matrix
		\(
		\widetilde{W}_{\ell-1}(\varepsilon)
		=
		\varepsilon I+\sum_{j=1}^{\ell-1}u_j u_j^T
		\)
		is symmetric positive definite, and hence nonsingular.
		Moreover,
		\(
		\widetilde{W}_{\ell}(\varepsilon)
		=
		\widetilde{W}_{\ell-1}(\varepsilon)+u_\ell u_\ell^T,
		\ \ell=1,\dots,Nm.
		\)
		Applying the rank-one determinant identity in multiplicative form
		to each step, we obtain
		\[
		\det\bigl(\widetilde{W}_{\ell}(\varepsilon)\bigr)
		=
		\det\bigl(\widetilde{W}_{\ell-1}(\varepsilon)\bigr)
		\left(
		1+u_\ell^T \widetilde{W}_{\ell-1}(\varepsilon)^{-1}u_\ell
		\right).
		\]
		Iterating this identity for \(\ell=1,\dots,Nm\) yields
		\[
		\det\bigl(\widetilde{W}_{Nm}(\varepsilon)\bigr)
		=
		\det\bigl(\widetilde{W}_0(\varepsilon)\bigr)
		\prod_{\ell=1}^{Nm}
		\left(
		1+u_\ell^T \widetilde{W}_{\ell-1}(\varepsilon)^{-1}u_\ell
		\right).
		\]
		Since
		\(
		\widetilde{W}_0(\varepsilon)=\varepsilon I,
		\
		\det(\varepsilon I)=\varepsilon^n,
		\)
		this gives (\ref{eq:gramian_regularized_product}).
		To prove (\ref{eq:gramian_pdet_limit}), let
		\(
		\lambda_1,\dots,\lambda_r>0
		\)
		be the nonzero eigenvalues of \(W_N\), where \(r=\rank(W_N)\le n\).
		Because \(W_N\) is symmetric positive semidefinite, its remaining
		eigenvalues are zero, and therefore
		\(
		\det(\varepsilon I+W_N)
		=
		\varepsilon^{\,n-r}\prod_{i=1}^r(\lambda_i+\varepsilon).
		\)
		Hence
		\(
		\varepsilon^{-(n-r)}\det(\varepsilon I+W_N)
		=
		\prod_{i=1}^r(\lambda_i+\varepsilon).
		\)
		Passing to the limit as \(\varepsilon\to 0\), we obtain
		\[
		\lim_{\varepsilon\to 0}
		\varepsilon^{-(n-r)}\det(\varepsilon I+W_N)
		=
		\prod_{i=1}^r\lambda_i
		=
		\pdet(W_N),
		\]
		which proves (\ref{eq:gramian_pdet_limit}).
		Since
		\(
		\widetilde{W}_{Nm}(\varepsilon)=\varepsilon I+W_N,
		\)
		combining (\ref{eq:gramian_pdet_limit}) with
		(\ref{eq:gramian_regularized_product}) gives
		\[
		\pdet(W_N)
		=
		\lim_{\varepsilon\to 0}
		\varepsilon^{-(n-r)}
		\det\bigl(\widetilde{W}_{Nm}(\varepsilon)\bigr)
		=
		\lim_{\varepsilon\to 0}
		\varepsilon^r
		\prod_{\ell=1}^{Nm}
		\left(
		1+u_\ell^T \widetilde{W}_{\ell-1}(\varepsilon)^{-1}u_\ell
		\right),
		\]
		which proves (\ref{eq:gramian_pdet_product_limit}).
		Finally, if \(\pdet(W_N)>0\), then taking logarithms in
		(\ref{eq:gramian_pdet_limit}) gives
		\[
		\log \pdet(W_N)
		=
		\lim_{\varepsilon\to 0}
		\left[
		\log \det(\varepsilon I+W_N) - (n-r)\log\varepsilon
		\right].
		\]
		Since \(\widetilde{W}_{Nm}(\varepsilon)=\varepsilon I+W_N\), this is
		exactly the stated formula.
	\end{pf}
	
	This representation provides a unified description of how the
	volume of the reachable subspace evolves under successive inputs.
	In particular, if $r = \rank(W_N)$, then $\pdet(W_N)$ is proportional
	to the squared volume of the $r$-dimensional ellipsoid induced by
	$W_N$ on the reachable subspace.
	
	This geometric interpretation can be made precise as follows.
	For the discrete-time system (2),
	the finite-horizon controllability Gramian $W_N$ characterizes
	the set of states reachable under energy-bounded inputs. More
	precisely, the set of states reachable in $N$ steps with unit
	input energy can be written as
	\(
	\mathcal{E}_N
	=
	\{x \in \operatorname{Im}(W_N):\ x^T W_N^D x \le 1\},
	\)
	and, in the nonsingular case,
	\(
	\mathcal{E}_N
	=
	\{x \in {\mathbb R}^n:\ x^T W_N^{-1} x \le 1\}.
	\)
	Thus, the Gramian defines an ellipsoid in the reachable subspace,
	whose principal axes and radii are determined by the eigenstructure
	of $W_N$. The normalization by $1$ in the inequality fixes the scale
	of the set.
	More precisely, if $W_N$ is nonsingular, then the ellipsoid
	\(
	\{x \in {\mathbb R}^n : x^T W_N^{-1} x \le 1\}
	\)
	has volume proportional to $\sqrt{\det(W_N)}$.
	In the singular case, the corresponding $r$-dimensional volume on the
	reachable subspace is proportional to $\sqrt{\pdet(W_N)}$, where
	$r = \rank(W_N)$.
	The regularized representation introduced above shows that this
	volume can be decomposed into successive contributions associated
	with individual input directions. Indeed, for $\varepsilon>0$ we have
	\[
	\det(\varepsilon I + W_N)
	=
	\varepsilon^n
	\prod_{\ell=1}^{Nm}
	\left(
	1 + u_\ell^T \widetilde W_{\ell-1}(\varepsilon)^{-1} u_\ell
	\right),
	\]
	where
	\(
	\widetilde W_{\ell-1}(\varepsilon)
	=
	\varepsilon I + \sum_{j=1}^{\ell-1} u_j u_j^T.
	\)
	Each factor
	\(
	1 + u_\ell^T \widetilde W_{\ell-1}(\varepsilon)^{-1} u_\ell
	\)
	quantifies the incremental contribution of the direction $u_\ell$
	relative to the current geometry encoded in the regularized Gramian.
	Directions already well represented produce only minor changes,
	whereas directions that are poorly represented lead to larger
	multiplicative contributions.
	Passing to the limit as $\varepsilon \to 0$, one recovers the
	pseudodeterminant through the normalization
	\(
	\pdet(W_N)
	=
	\lim_{\varepsilon\to 0}
	\varepsilon^{-(n-r)} \det(\varepsilon I + W_N),
	\)
	which shows that only those directions that expand the reachable
	subspace contribute nontrivially in the limit.
	Taken together, these results show that determinant-based
	representations provide a unified framework for analyzing
	information accumulation, uncertainty reduction, and reachability
	in linear systems. In particular, they establish a direct link
	between the geometry of the reachable set and the incremental
	contribution of individual input directions.
	
	To further illustrate the geometric mechanism underlying successive
	rank-one updates, we visualize the evolution of the Gramian-induced
	ellipsoid; see Fig.~\ref{fig:gramian_ellipsoid}. Starting from the
	regularized matrix $W_0=\varepsilon I$, the associated ellipsoid
	\(
	\{x \in {\mathbb R}^n : x^T W_0^{-1} x \le 1\}
	\)
	represents the set of states reachable with bounded input energy.
	The normalization by $1$ fixes the energy level and ensures that
	the volume of the set is directly related to the determinant.
	In this case, the ellipsoid reduces to the Euclidean ball
	\(
	\{x \in {\mathbb R}^n : \|x\| \le \sqrt{\varepsilon}\},
	\)
	that is, a circle of radius $\sqrt{\varepsilon}$ in two dimensions.
	This follows from the identity $W_0^{-1} = \varepsilon^{-1} I$, which
	yields $x^T W_0^{-1} x = \|x\|^2 / \varepsilon \le 1$.
	Each subsequent update of the form $u_\ell u_\ell^T$ deforms and expands
	this ellipsoid. The first update introduces a dominant direction aligned
	with $u_1$, resulting in an elongated ellipsoid. The effect of the second
	update depends critically on its alignment with the existing structure:
	if $u_2$ is aligned with $u_1$, the contribution is small, whereas a
	component orthogonal to $u_1$ produces a significant expansion.
	This behavior is consistent with the quadratic form
	\(
	u_2^T W_1^{-1} u_2,
	\)
	which quantifies the contribution of $u_2$ relative to the current
	Gramian. In particular, directions that are poorly represented in
	$W_1$ lead to larger values of this expression and therefore induce
	a stronger expansion of the ellipsoid.
	This visualization highlights the multiplicative structure derived
	in the paper: the volume of the reachable set evolves through
	successive directional contributions, with each update expanding
	the ellipsoid according to its novelty relative to the current
	subspace.
	More precisely, the Gramian $W_N$ defines an ellipsoid whose volume
	is proportional to $\sqrt{\pdet(W_N)}$, so that the pseudodeterminant
	captures the squared volume of the reachable set restricted to its
	intrinsic subspace.
	\begin{figure*}[t]
		\centering
		\includegraphics[width=0.98\textwidth]{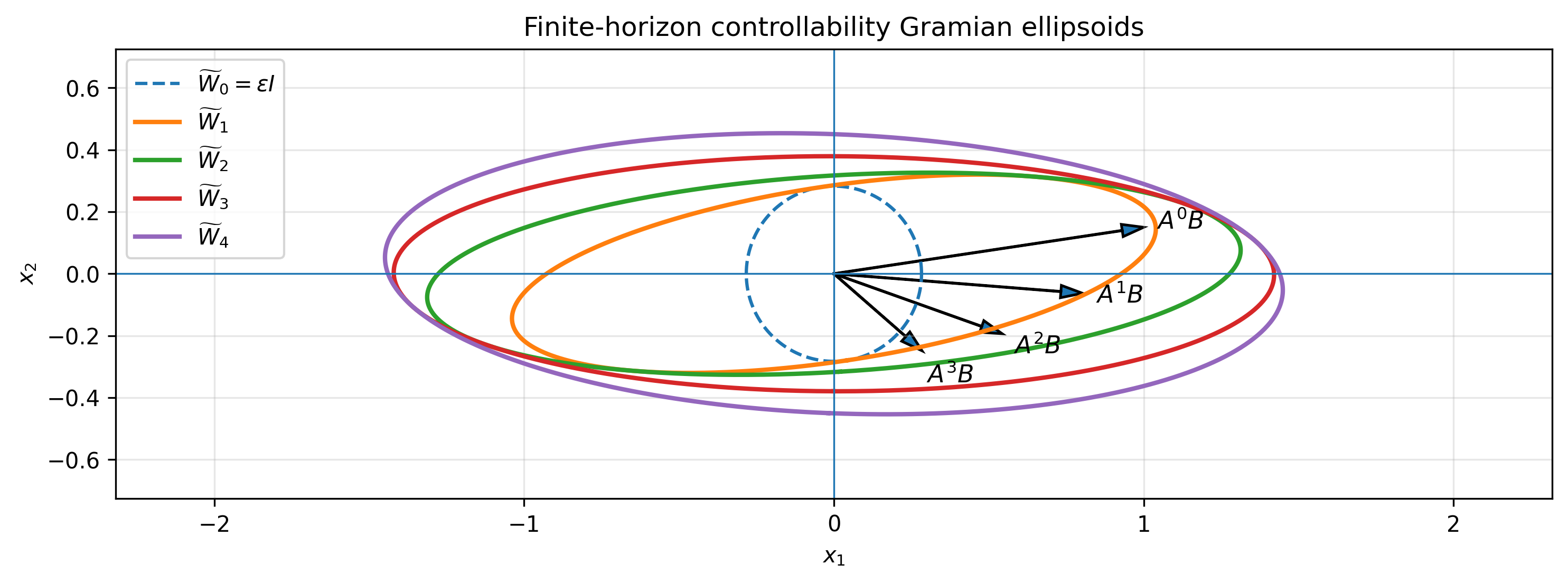}
		\caption{Evolution of the controllability Gramian ellipsoid for the discrete-time system
			$x_{t+1}=Ax_t+Bu_t$ with
			$A=\left[\begin{array}{cc}0.72 & 0.55\\ -0.18 & 0.78\end{array}\right]$,
			$B=\left[\begin{array}{c}1.0\\ 0.15\end{array}\right]$,
			and horizon $N=4$.
			Starting from the regularized matrix $W_0=\varepsilon I$, the ellipsoid
			$\{x : x^T W_k^{-1} x \le 1\}$ evolves under successive rank-one
			contributions $u_\ell u_\ell^T$, where, for $m=1$, the update directions
			are given by $u_\ell = A^{\ell-1} B$, $\ell=1,\dots,N$.
			Each update expands the ellipsoid in directions not yet represented in
			the Gramian, while contributions aligned with existing directions induce
			only marginal growth. This provides a geometric interpretation of the
			multiplicative structure of $\pdet(W_k)$ and the incremental expansion
			of the reachable set.}
		\label{fig:gramian_ellipsoid}
	\end{figure*}
	\subsection{Directional novelty and robustness under perturbations}
	To further extend this geometric picture, it is instructive to consider
	the effect of perturbations or modeling uncertainty on the evolution
	of the Gramian. Suppose that the update directions are not exactly
	given by $u_\ell = A^i B$, but by
	\(
	u_\ell = A^i B + w_\ell,
	\)
	where $w_\ell \in \mathbb{R}^n$ represents a small disturbance,
	modeling error, or unmodeled input component. The corresponding
	Gramian then takes the form
	\(
	W_N = \sum_{\ell=1}^{Nm} (u_\ell + w_\ell)(u_\ell + w_\ell)^T,
	\)
	which can be decomposed into nominal, perturbation, and interaction
	terms. While the nominal component $\sum u_\ell u_\ell^T$ reflects
	the structure imposed by the system dynamics, the additional terms
	introduce new directions that are not necessarily aligned with the
	deterministic evolution.
	From a geometric viewpoint, these perturbations act as a mechanism
	that enriches the span of reachable directions. In particular, even
	small components $w_\ell$ that are transverse to the current subspace
	can produce non-negligible contributions in the quadratic form
	$u_\ell^T \widetilde{W}_{\ell-1}^{-1} u_\ell$, thereby increasing
	the corresponding multiplicative factors in the determinant evolution.
	As a consequence, the associated ellipsoid expands not only along the
	dominant directions dictated by $A^i B$, but also in directions that
	would remain unexcited in the nominal setting. This leads to a more
	isotropic growth of the reachable set and, in turn, to a faster increase
	of $\pdet(W_N)$.
	Importantly, this behavior is fully consistent with the determinant
	dynamics developed in this paper. The multiplicative structure of the
	regularized determinant shows that each update contributes through a
	factor of the form
	\(
	1 + u_\ell^T \widetilde{W}_{\ell-1}^{-1} u_\ell,
	\)
	which quantifies the novelty of the direction $u_\ell$ relative to the
	current geometry. Perturbations effectively increase this novelty by
	injecting components outside the existing span, and therefore lead to
	a systematic amplification of the cumulative determinant growth. In
	this sense, the proposed framework provides a quantitative mechanism
	for distinguishing between reinforcement of existing directions and
	activation of genuinely new ones.
	This interpretation reveals a direct link between determinant-based
	quantities and robustness properties of dynamical systems. In
	particular, the pseudodeterminant emerges as an intrinsic measure of
	the effective system volume, capturing not only the nominal reachable
	subspace but also its expansion under uncertainty. Such effects are
	closely related to phenomena encountered in stochastic reachability,
	covariance inflation, and information propagation in uncertain systems.
	Taken together, these observations indicate that the proposed
	determinant-based framework extends naturally beyond deterministic
	settings and provides a unified tool for analyzing how structure,
	excitation, and uncertainty interact in shaping the geometry of
	reachable sets. This perspective opens several directions for future
	research, including extensions to stochastic systems, robustness
	analysis, and the design of inputs that explicitly maximize the growth
	of intrinsic system volume.

	\section{Conclusions}\label{sec:conclusion}
	
	This paper has developed a unified framework for determinant identities
	under finite-rank perturbations of square matrices that remains valid
	without invertibility assumptions. By exploiting an adjugate-based
	representation, we obtained explicit, non-asymptotic formulas that extend
	classical determinant relations to singular matrices.
	A central contribution is the identification of recursive and
	multiplicative structures governing the evolution of determinant and
	log-determinant quantities under successive rank-one updates. These
	results reveal that determinant-based quantities admit a decomposition
	into incremental contributions associated with individual directions,
	providing a precise structural interpretation of determinant evolution.
	The proposed framework enables a consistent treatment of singular systems
	via the Drazin inverse and the pseudodeterminant, leading to closed-form
	identities that isolate the contribution of the nonzero spectrum. This
	yields a natural extension of the matrix determinant lemma to the
	rank-deficient setting.
	Beyond the algebraic results, the paper establishes a direct connection
	between matrix perturbation theory and system-theoretic concepts. In
	particular, the pseudodeterminant of controllability Gramians is shown to
	admit a multiplicative decomposition that quantifies the incremental
	expansion of the reachable subspace under successive inputs. This
	provides a unified interpretation of information accumulation,
	uncertainty reduction, and reachability in linear systems.

\end{document}